\documentclass{amsart}

\usepackage{amsmath, amssymb,amsfonts}
\usepackage[all]{xy}
\usepackage{tikz}
\usepackage{endnotes}

\title{On prime degree isogenies between K3 surfaces}

\author{Samuel Boissi\`ere}
\address{Samuel Boissi\`ere, Universit\'e de Poitiers, 
Laboratoire de Math\'ematiques et Applications, 
 T\'el\'eport 2, Boulevard Marie et Pierre Curie
 BP 30179, 86962 Futuroscope Chasseneuil Cedex, France}
\email{samuel.boissiere@math.univ-poitiers.fr}
\urladdr{http://www-math.sp2mi.univ-poitiers.fr/$\sim$sboissie/}

\author{Alessandra Sarti}
\address{Alessandra Sarti, Universit\'e de Poitiers, 
Laboratoire de Math\'ematiques et Applications, 
 T\'el\'eport 2, Boulevard Marie et Pierre Curie,
 BP 30179, 86962 Futuroscope Chasseneuil Cedex, France}
\email{sarti@math.univ-poitiers.fr}
\urladdr{http://www-math.sp2mi.univ-poitiers.fr/$\sim$sarti/}

\author{Davide Cesare Veniani}
\address{Institut f\"ur Algebraische Geometrie,
Leibniz Universit\"at Hannover,
Wel\-fen\-garten 1, 30167 Hannover, Germany}
\email{veniani@math.uni-hannover.de}

\date{\today}

\newcommand{\ie}{{\it i.e. }}
\newcommand{\eg}{{\it e.g. }}
\newcommand{\resp}{{\it resp. }}
\newcommand{\viceversa}{{\it vice versa}}

\newcommand{\IC}{\mathbb C}
\newcommand{\IF}{\mathbb F}
\newcommand{\IK}{\mathbb K}
\newcommand{\IN}{\mathbb N}
\newcommand{\IP}{\mathbb P}
\newcommand{\IQ}{\mathbb Q}
\newcommand{\IR}{\mathbb R}
\newcommand{\IZ}{\mathbb Z}

\newcommand{\lra}{\longrightarrow}

\DeclareMathOperator{\Hom}{Hom}

\DeclareMathOperator{\disc}{disc}
\DeclareMathOperator{\rk}{rk}
\DeclareMathOperator{\res}{res}
\DeclareMathOperator{\Km}{Km}

\newcommand{\Tr}{{\rm{T}}}
\newcommand{\TransQY}{\Tr_{Y,\IQ}}
\newcommand{\TransQX}{\Tr_{X,\IQ}}

\newcommand{\NS}{{\rm{NS}}}

\newcommand{\HH}{{\rm{H}}}

\newcommand{\Hdg}{{\rm{Hdg}}}

\newcommand{\id}{{\rm{id}}}

\theoremstyle{plain}
\newtheorem{theorem}{Theorem}[section]

\newtheorem{proposition}[theorem]{Proposition}
\newtheorem{corollary}[theorem]{Corollary}
\theoremstyle{definition}

\theoremstyle{remark}
\newtheorem{remark}[theorem]{Remark}

\begin{document}

\begin{abstract} 
We classify prime order isogenies between algebraic K3 surfaces whose rational transcendental Hodges structures are not isometric. The morphisms of Hodge structures induced by these isogenies are 
correspondences by algebraic classes on the product fourfolds; however, they do not satisfy the hypothesis of the well-known Mukai--Nikulin theorem.
As an application we describe isogenies obtained from K3 surfaces with an action of a symplectic
automorphism of prime order.
\end{abstract}

\maketitle

\section{Introduction}

Let $X$ and $Y$ be two complex projective K3 surfaces. The Hodge conjecture 
predicts that every Hodge class 
$Z\in \HH^4(Y\times X, \IQ)\cap \HH^{2,2}(Y\times X)$ is algebraic (see \eg \cite[\S6.4]{Beauville}) . 
By K\"unneth decomposition, the question reduces to the study of the Hodge
classes~$Z$ in $\HH^2(Y,\IQ)\otimes \HH^2(X,\IQ)$. 
Using the unimodularity of the K3 lattice, this data is equivalent to a morphism $\phi$ of weight two
rational Hodge structures between $\HH^2(Y,\IQ)$ and $\HH^2(X,\IQ)$, where $\phi$
is nothing else than the correspondence induced by the class $Z$. Since the rational transcendental
lattice of a K3 surface has an irreducible Hodge structure of weight two, the morphism $\phi$
decomposes as the sum of an Hodge morphism $\phi_{\Tr}$ 
between the rational transcendental lattices $\TransQY$ and~$\TransQX$, which is either trivial or an isomorphism, and a morphism $\phi_{\NS}$
between the rational N\'eron--Severi lattices $\NS(Y)_\IQ$ and $\NS(X)_\IQ$. The morphism $\phi_\NS$
is automatically induced by an algebraic correspondence by the Lefschetz theorem on hyperplane sections,
so the question finally reduces to the algebraicity of the correspondence $\phi_\Tr$.

There is a wide literature on the topic, especially when $X=Y$ 
(see \eg \cite{ chen,galluzzilombardo, inose, inoseisog,ma,morrisonalg,morrisonisog,mukai,Nikulincorresp,Nikulinrational}).
One of the most famous results is due to Mukai~\cite[Corollary~1.10]{mukai} and Nikulin~\cite{Nikulincorresp}:
if $\rk\NS(Y)\geq 5$ and if $\phi_\Tr$ is an isometry, then $\phi_\Tr$~is algebraic. A generalization of this result
to any projective K3 surfaces has been recently announced by Buskin~\cite{Buskin}. Here the term
``isometry'' means isometry of rational quadratic spaces.
However the Hodge conjecture predicts that any homomorphism of rational transcendental Hodge structures 
of K3 surfaces is induced by an algebraic cycle, without requiring that it is an isometry. 
The question addressed in this paper is thus the following:
can \emph{non-isometric} rational transcendental Hodges structures of K3 surfaces be related 
by an algebraic correspondence?
An example was given by van Geemen and the second author in~\cite[Proposition~2.5]{vanGeemen-Sarti} where the
Hodge isomorphism between the rational transcendental lattices is induced by an isogeny of degree two 
between a K3 surface $X$ and the minimal
resolution~$Y$ of its quotient by a symplectic involution. This example was a motivation
for us to study more in general the Hodge isomorphisms induced by isogenies between K3 surfaces in the context
of the Hodge conjecture.

As noted by Morrison~\cite{morrisonalg}, the term ``isogeny'', when applied to K3 surfaces, has
several conflicting definitions in the literature. In this paper we follow the definition 
of Inose~\cite{inose}:
an isogeny between two complex projective surfaces $X$~and~$Y$ is a rational map of finite degree 
$\gamma\colon X\dashrightarrow Y$. Following Inose~\cite{inose}, the rational map~$\gamma$ induces an 
isomorphism of rational Hodge stuctures $\gamma^\ast\colon\HH^2(Y,\IQ)\to \HH^2(X,\IQ)$ 
given by an algebraic correspondence, whose component $\gamma^\ast_\Tr\colon\TransQY\to\TransQX$ can 
be described in geometric terms; 
in particular, one has $\rk\Tr_X=\rk\Tr_Y$. The morphism~$\gamma^\ast_\Tr$ is not an isometry but it
 is a dilation
with a scale factor $p$ (see Section~\ref{prelim}). Equivalently, $\gamma^\ast_\Tr$ is an isometry between $\TransQY(p)$ (where the intersection form
is multiplied by $p$) and $\TransQX$. If $\TransQY(p)$ happens to be isometric to $\TransQY$, then $\gamma^\ast_\Tr$
is closely related to an isometry of rational Hodge structure (a \emph{Hodge isometry} for short): roughly speaking, $\gamma^\ast_\Tr$
becomes an isometry after a (nonisometric) base change.
The main result of this paper is a precise condition on $\Tr_Y$ 
so that there exists \emph{no} isometry (in particular no Hodge isometry) between $\TransQY$ and $\TransQX$.

\begin{theorem}\label{witt}
Let $\gamma\colon X\dashrightarrow Y$ be an isogeny of prime order $p$ between complex projective K3 surfaces $X$ and $Y$. Then $\rk \TransQY=\rk \TransQX=:r$ and
\begin{enumerate}
\item\label{item1} If $r$ is odd, there exist no isometry between  $\TransQY$  and $\TransQX$.
\item\label{item2} If $r$ is even, there exists an isometry between $\TransQY$ and $\TransQX$ 
if and only if $\TransQY$ is isometric to $\TransQY(p)$. This property is equivalent to the following:
\begin{enumerate}
\item\label{item3} If $p=2$: for every prime number $q$ congruent to $3$ or $5$ modulo $8$, the $q$-adic valuation $\nu_q(\det \Tr_Y)$ is even.
\item\label{item4} If $p>2$: for every prime number $q >2$, $q \neq p$, such that $p$ is not a square in $\IF_q$, 
the number $\nu_q(\det \Tr_Y)$ is even and the following equation holds in $\IF_p^*/(\IF_p^*)^2$:
\begin{equation*}
\res_p (\det \Tr_Y) = (-1)^{\frac{n(n-1)}{2} + \nu_p(\det \Tr_Y)},
\end{equation*}
where $\res_p(\det \Tr_Y)$ is the residue of $\det \Tr_Y$ modulo $p$ (see definition in Section \ref{isoge}).
\end{enumerate}
\end{enumerate}
\end{theorem}
The proof is given in Section \ref{isoge} using several results coming from lattice theory 
and the theory of Witt groups that are recalled in Section~\ref{prelim}.
This result suggests that there is a large family of isogenies $\gamma$ that produce algebraic correspondences
between nonisometric rational transcendental
Hodge structures of K3 surfaces, so that the Hodge conjecture holds in this context without Mukai--Nikulin's isometric condition. 
To construct concrete examples one can generalize the example given in~\cite[Proposition~2.5]{vanGeemen-Sarti}
to K3 surfaces with a symplectic automorphism of prime order:

\begin{corollary} \label{main}
Let $\sigma$ be a symplectic automorphism of prime order $p$ on a complex projective K3 surface $X$
and let $Y$ the minimal resolution of the quotient~$X/\sigma$, which is a K3 surface with a degree $p$
isogeny $\gamma\colon X\dashrightarrow Y$.
If $\rk\Tr_X$ is odd, then there exists no isometry between $\TransQY$ and $\TransQX$.
\end{corollary}

Note that the condition $\rk \Tr_{X}$ odd corresponds to the generic surface in the moduli space of K3 
surfaces with a symplectic automorphism of order $p$ (see Section~\ref{sympl}). 
This corollary is a direct consequence of Theorem~\ref{witt} since it is a special
case of Inose's construction. In Section~\ref{symplnoniso} we give a different and
completely self contained proof (the necessary background is given in Section~\ref{sympl}). This second proof is of independent interest since it uses deep
geometric properties of symplectic automorphisms on K3 surfaces: the key is to show
the existence of a  commutative diagram of blow-ups and contractions from which the properties of 
$\gamma^\ast_\Tr$ and consequently of the rational quadratic spaces $\TransQY$ and $\TransQX$  can be deduced by geometric
 arguments without using Witt theory. This diagram has been studied by Morrison~\cite{morrison} and van Geemen--Sarti~\cite{vanGeemen-Sarti} for $p=2$ and by Tan~\cite{tan} for $p=3,5$. To our knowledge, the construction of the diagram for $p=7$ is new.

To conclude this paper, we discuss in Section \ref{exos} some classical and interesting examples of 
isogenies.

{\bf Acknowledgement.} We thank Bert van Geemen, Xavier Roulleau and Matthias Sch\"utt for helpful discussions.


\section{Preliminaries}\label{prelim}


\subsection{Lattices}\label{lattices} 

A \emph{lattice} $L$ is a free $\IZ$-module of finite rank equipped with a nondegenerate integral quadratic form with integer values. The determinant of the bilinear form (computed with respect to any basis) is called the \emph{discriminant} of the lattice $L$ and is denoted $\disc L$.

If $M \subset L$ is a sublattice of the same rank, the quotient $L/M$ has finite order and we have
\begin{equation}\label{formule_disc}
    \disc M = \disc L \cdot [L:M]^2.
\end{equation}

A sublattice $M \subset L$ is \emph{primitive} if $L/M$ is free. Then $M \oplus M^\perp$ has finite index in~$L$ and
\begin{equation} \label{eq:disc-perp}
    \disc M \cdot \disc M^\perp = \disc L \cdot [L:M\oplus M^\perp]^2.
\end{equation}

For any lattice $L$, we denote by $L(n)$ the lattice $L$ whose quadratic form is multiplied by $n\in\IZ^\ast$ and by $L_\IQ$ the $\IQ$-vector space $L \otimes \IQ$ with the induced rational quadratic form.

\subsection{Classification of rational quadratic forms.}

We recall some classical material for which we refer to~\cite{scharlau}.
Two quadratic forms $\varphi$ and~$\varphi'$ with coefficients in a field~$\IK$ are \emph{equivalent} if there exists a linear isomorphism between them. Equivalently, their associated matrices $Q$ and $Q'$ in any bases are \emph{congruent}, \ie there exists an invertible matrix $A$ with coefficients in $\IK$
such that $Q'= A^\top Q A$. In particular, when these quadratic forms are nondegenerate this implies that 
\begin{equation}\label{trivial}
\det Q \equiv \det Q'\mod (\IQ^*)^2.
\end{equation}
This basic property will be very useful in the sequel to detect which dilations associated to isogenies cannot be related to an isometry.

Given $a_1,\ldots,a_n\in\IK$ we denote by $\langle a_1,\ldots,a_n\rangle$ the quadratic form
\begin{equation}\label{diagonal}
(x_1,\ldots,x_n)\mapsto a_1x_1^2+\cdots+a_nx_n^2.
\end{equation}
If the characteristic of $\IK$ is different from two, any equivalence class of a quadratic form
contains a representative in diagonal form (\ref{diagonal}).

Any regular (equivalently nondegenerate) quadratic form $\varphi$ decomposes as
$$
\varphi \cong \langle 1, -1 \rangle^{\oplus \ell} \oplus \varphi_a
$$
where $\ell\in\IN$ is uniquely determined and $\varphi_a$ is the anisotropic part of $\varphi$, well-defined
up to equivalence. Two regular quadratic forms $\varphi$ and $\psi$, not necessarily of the same dimension,
are \emph{Witt-equivalent} if their anisotropic parts $\varphi_a$ and $\psi_a$ are equivalent.
The Witt-equivalence class of $\varphi$, \resp $\langle a_1, \ldots, a_n \rangle$, is denoted
by $[\varphi]_W$, \resp $\langle a_1, \ldots, a_n \rangle_W$. 
The orthogonal sum operation $\oplus$ defines the group law of the \emph{Witt group} $W(\IK)$ of Witt-equivalence classes 
of regular quadratic forms on $\IK$.

The \emph{discriminant} of a rational quadratic form $\varphi$ of dimension $n$ is
$$
\Delta(\varphi) := (-1)^{\frac {n(n-1)} 2} \det(\varphi) \in \IK/(\IK^*)^2,
$$
it depends only on the Witt-equivalence class of~$\varphi$. For any $r \in \IQ^\ast$ and any prime number $q$, we denote by $\nu_q(r)$ the $q$-adic valuation of $r$. 
Writing $r = q^{\nu_q(r)} \frac st$
where $s$ and $t$ are integers prime to $q$, the residue of $r$ modulo $q$ is by definition
$$
\res_q(r) := \bar s \bar t^{-1} \in \IF_q^*
$$
and then we set
$$
\partial_q (r) := \begin{cases} 
                            0 & \text{if $\nu_q(r)$ is even} \\
                    \res_q(r) & \text{if $\nu_q(r)$ is odd.}
                  \end{cases}
$$
For any $q$, there exists a unique group homomorphism
$$
\bar \partial_q \colon W(\IQ) \rightarrow W(\IF_q)
$$
such that
$$
\bar\partial_q (\langle a_1,\ldots, a_n \rangle_W) = \langle \partial_q (a_1),\ldots,\partial_q(a_n) \rangle_W \quad \forall a_1,\ldots,a_n\in \IQ^*.
$$
We recall the following classical results of Witt theory:
\begin{proposition} \label{prop:QQ-equivalence}
Two regular rational quadratic forms $\varphi$ and $\psi$ are equivalent if and only if they have the same signature 
 over $\IR$ and if for every prime number $q$:
$$
\bar \partial_q([\varphi]_W) = \bar \partial_q([\psi]_W).
$$
\end{proposition}

\begin{proposition} \label{prop:finite-field}
 Two regular quadratic forms on a finite field are Witt-equivalent if and only if their dimensions have
same parity and their discriminants are equal.
\end{proposition}


\subsection{Hodge theory for products of K3 surfaces}\label{hodge}

Let $M$ be a smooth complex compact K\"ahler manifold. Its cohomology with complex coefficients admits a Hodge decomposition
$$
\HH^k(M,\IC)=\bigoplus_{p+q=k} \HH^{p,q}(M)\quad\forall k
$$
where $\HH^{p,q}(M)$ is the space of classes of differential $(p,q)$-forms. Every smooth codimension~$k$ subvariety $Z\subset X$ admits a fundamental class $[Z]$ in cohomology and it is
easy to see that $[Z]\in\HH^{k,k}(M)$. Conversely, denote
$$
\Hdg^k(M):=\HH^{2k}(M,\IQ)\cap \HH^{k,k}(M)
$$
the group of Hodge classes of degree $2k$ on $M$. The Hodge conjecture predicts that
if $M$ is projective, then any Hodge class is a linear combination with rational coefficients of cohomology classes of algebraic subvarieties of $M$.

Let $X$ be a projective K3 surface. The intersection product, denoted $\langle-,-\rangle_X$ in the sequel, gives its second cohomology space $\HH^2(X,\IZ)$ the structure of an even unimodular lattice of rank $22$ and signature $(3,19)$, isometric to $U^3\oplus E_8^2$ where $U$ is the hyperbolic plane and $E_8$ is the unique even unimodular negative definite lattice of rank~$8$. The $\IQ$-vector space $\HH^2(X ,\IQ)=\HH^2(X,\IZ)\otimes \IQ$ carries a weight two Hodge structure
$$
\HH^2(X,\IQ)\otimes \IC=\HH^2(X,\IC)=\HH^{2,0}(X)\oplus \HH^{1,1}(X)\oplus \HH^{0,2}(X)
$$ 
where the space $\HH^{2,0}(X)$ of global holomorphic $2$-forms is one-dimensional,
$\HH^{1,1}(X)$ is $20$-dimensional and $\HH^{0,2}(X)=\overline{\HH^{2,0}(X)}$.
Inside $\HH^2(X,\IZ)$ there are two distinguished sublattices: the {\it N\'eron Severi lattice} $\NS(X):=\HH^{1,1}(X)\cap\HH^2(X,\IZ)$ and the \emph{transcendental lattice} $\Tr_X:=\NS(X)^\perp$. Since $X$ is projective, its hyperplane section is an integral $1$-codimensional cycle hence $\rho(X):=\rk\NS(X)\geq 1$. By the Hodge index theorem $\NS(X)$ has signature $(1,\rho(X)-1)$.
The rational transcendental lattice $\Tr_{X,\IQ}:=\Tr_X\otimes \IQ$ admits a weight two Hodge structure $\Tr_{X,\IC}=\Tr^{2,0}_X\oplus\Tr^{1,1}_X\oplus \Tr^{0,2}_X$ given by $\Tr_X^{p,q}:=\HH^{p,q}(X)\cap \Tr_{X,\IC}$,
where $\Tr_X^{2,0}$ and $\Tr_X^{0,2}$ are one-dimensional and $\Tr_X^{1,1}$ contains no nonzero rational class. It follows easily that $\Tr_{X,\IQ}$ is an irreducible
Hodge structure (see \eg \cite[Lemma~1.7]{vanGeemenMichigan}).

Assume that $M=Y\times X$ is a product of two projective K3 surfaces $X$ and~$Y$.
By Lefschetz theorem on hyperplane sections, $\Hdg^1(M)=\NS(M)_\IQ$ so it contains only algebraic classes and by the Hard Lefschetz theorem $\Hdg^3(M)$ thus contains also only algebraic classes. So Hodge conjecture remains
open for $\Hdg^2(M)$. By K\"unneth decomposition, the question is non trivial only for classes living in $\HH^{2,2}(M)\cap (\HH^2(Y,\IQ)\otimes\HH^2(X,\IQ))$. Since $\HH^2(Y,\IQ)$ is unimodular, we have 
$$
\HH^2(Y,\IQ)\otimes\HH^2(X,\IQ)\cong \HH^2(Y,\IQ)^\vee\otimes\HH^2(X,\IQ)\cong\Hom(\HH^2(Y,\IQ),\HH^2(X,\IQ)).
$$
Denoting by $\pi_Y$, \resp $\pi_X$, the projection of $Y\times X$ onto $Y$, \resp $X$,
the last isomorphism identifies a class $[Z]\in \HH^2(Y,\IQ)\otimes\HH^2(X,\IQ)$
with the correspondence $x\mapsto \pi_{X\ast}(\pi_Y^\ast x\cdot[Z])$.
It is easy to check that this correspondence is compatible with the Hodge structures if and only if $[Z]\in \HH^{2,2}(M)$.

Decomposing $\HH^2(X,\IQ)=\NS(X)_{\IQ}\oplus \TransQX$, and similarly for $Y$,
since the Hodge structure of the rational N\'eron--Severi lattice has weight one and those of the rational transcendental lattice is irreducible of weight two, there
are no nonzero homomorphisms of Hodge structures between them. It follows that any
morphism of Hodge structure $\phi\in\Hom_\Hdg(\HH^2(Y,\IQ),\HH^2(X,\IQ))$ decomposes as $\phi=\phi_\NS+\phi_\Tr$ where $\phi_\NS\in \Hom_\Hdg(\NS(Y)_\IQ, \NS(X)_\IQ)$
and $\phi_\Tr\in \Hom_\Hdg(\TransQY, \TransQX)$ is either zero or an
isomorphism. The morphism $\phi$ is the correspondence by a class $[Z]\in\HH^{2,2}(M)\cap(\HH^2(Y,\IQ)\otimes\HH^2(X,\IQ))$. The morphism $\phi_\NS$ 
is the correspondence by the component of $[Z]$ in $\NS(Y)_\IQ\otimes\NS(X)_\IQ$, which is automatically algebraic. 
Finally the study of the algebraic cycles in $\Hdg^2(Y\times X)$ reduces to the study of morphisms of rational Hodge structures between $\TransQY$ and $\TransQX$ that come from algebraic correspondences. Conversely, any such morphism extends to an element of $\Hom_\Hdg(\HH^2(Y,\IQ),\HH^2(X,\IQ))$ by projection and inclusion. We can thus formulate the Hodge conjecture for a product of complex projective K3 surfaces $Y\times X$ as follows: Is every morphism in $\Hom_{\Hdg}(\TransQY,\TransQX)$ given by an algebraic correspondence?


\subsection{Correspondences induced by isogenies}\label{corresp} 

Let $\gamma\colon X\dashrightarrow Y$ be an isogeny between two complex projective K3 surfaces $X$ and $Y$, \ie a rational map of finite degree $n$. By elimination of indeterminacy there exists a birational morphism $\beta\colon \widetilde{X}\to X$ and a morphism $\widetilde\gamma\colon\widetilde X\to Y$ which resolve the indeterminacies of $\gamma$, \ie $\widetilde{\gamma}=\gamma\circ\beta$:
\begin{equation}\label{resol}
\xymatrix{\widetilde X\ar[rd]^{\widetilde\gamma}\ar[d]_\beta\\ X\ar@{-->}[r]_\gamma & Y}
\end{equation}
The composition of the correspondences by the graphs $\Gamma_\beta$ of $\beta$ and $\Gamma_{\widetilde\gamma}$ of $\widetilde\gamma$ produce a morphism
$$
\gamma^\ast:=[\Gamma_{\beta}]_\ast\circ[\Gamma_{\widetilde\gamma}]^\ast\colon \HH^\ast(Y,\IZ)\to\HH^\ast(X,\IZ)
$$ 
which is easily identified with the
correspondence $[\Gamma]^\ast$ by the graph
$$
\Gamma:=\{(\beta(x),\widetilde\gamma(x))\,|\,x\in\widetilde X\}\subset X\times Y
$$
which gives
an algebraic class in $\Hdg^2(Y\times X)$.
Symmetrically, the composition of correspondences
$$
\gamma_\ast:=[\Gamma_{\widetilde\gamma}]_\ast\circ[\Gamma_{\beta}]^\ast\colon \HH^\ast(X,\IZ)\to\HH^\ast(Y,\IZ)
$$
is the correspondence $[\Gamma]_\ast$. Once restricted to the second cohomology groups, the morphisms $\gamma_\ast$ and $\gamma^\ast$ are adjoint to each other for the intersection product.
 
Consider the composed morphism
$$
\gamma_\ast\gamma^\ast=[\Gamma_{\widetilde\gamma}]_\ast\circ[\Gamma_{\beta}]^\ast\circ[\Gamma_{\beta}]_\ast\circ[\Gamma_{\widetilde\gamma}]^\ast=\widetilde\gamma_\ast\beta^\ast\beta_\ast\widetilde\gamma^\ast.
$$
Note that $\widetilde\gamma$ is generically finite of degree $n$ and that since $\beta$ is a sequence of blow-ups, $\beta^\ast\colon \HH^2(X,\IZ)\to\HH^2(\widetilde X,\IZ)$ is an inclusion. To study the restriction of 
$\gamma_\ast\gamma^\ast$ to $\Tr_Y$ we need a geometric description of the restriction of $\gamma^\ast$ to $\Tr_Y$, which is explained by Inose~\cite[\S 1]{inose} as follows. Fix a set of curves $C_i$, \resp $D_j$, whose classes generate $\NS(\widetilde X)$, \resp $\NS(Y)$, and let $N$ be a reducible curve in~$Y$ which contains in its support all the images $\widetilde{\gamma}(C_i)$ and the curves~$D_j$. Take an element $t\in\Tr_Y$
and denote by $\tau\in \HH_2(Y,\IZ)$ its Poincar\'e dual. By assumption we have $\int_{N_i} t=0$ for any irreducible component $N_i$ of~$N$, hence by a standard transversality argument we can take the curve $N$ to meet $\tau$ transversally, so that here $\tau\cap N=\emptyset$. It follows from the definition that $\gamma^\ast(t)$ is the Poincar\'e dual of $\beta(\widetilde\gamma^{-1}(\tau))$. As a consequence, since $\widetilde\gamma$ is generically finite of degree $n$ we have
$$
(\gamma_\ast\gamma^\ast)_{|\Tr_Y}=n\,\id_{\Tr_Y}
$$
from which follows by adjunction that
$$
\langle\gamma_\Tr^*(x),\gamma_\Tr^*(y)\rangle_X=\langle x,\gamma_{\Tr\ast}\gamma_\Tr^*(y)\rangle_Y=n\,\langle x,y\rangle_Y \qquad\forall x,y\in\Tr_Y.
$$
In particular $\rk(\Tr_Y)=\rk(\Tr_X)$ and $\gamma_\Tr^*$ is a dilation with scale factor $n$. We denote also by $\gamma^\ast_\Tr\colon\TransQY\to\TransQX$ its $\IQ$-linear
extension, which is thus an isomorphism of rational Hodge structures induced
by an algebraic cycle, but not an isometry. So the Hodge conjecture is verified for $\gamma^\ast_\Tr$ even if it does not satisfy the hypothesis of the Mukai--Nikulin theorem. Our objective in this paper is to study those isogenies between projective K3 surfaces whose transcendental rational Hodge structures are not isometric at all.


\section{Proof of Theorem~\ref{witt}}\label{isoge}

The proof of Theorem~\ref{witt} is decomposed into four slightly more general Propositions of independent interest: assertion~(\ref{item1}) is proven in Proposition~\ref{isog}, assertion~(\ref{item2}) is proven in Proposition~\ref{subsp} and the equivalences of assertion ~(\ref{item2}) with either assertion~(\ref{item3}) or assertion (\ref{item4}) is proven in Propositions~\ref{due} \and \ref{primo}.

\begin{proposition}\label{isog}
Let $\gamma\colon X\dashrightarrow Y$ be an isogeny of degree $n$ between two complex projective K3 surfaces. Assume that $n$ is not a square in $\IQ^*$ and that $\rk\Tr_X$ is odd. Then there exist no isometry between $\TransQY$ and $\TransQX$.
\end{proposition}

\begin{proof}
As explained in Section \ref{corresp}, the morphism $\gamma^\ast_\Tr\colon\Tr_Y\to\Tr_X$ is a dilation with scale factor $n$, hence
$\gamma^\ast_\Tr\colon\Tr_Y(n)\hookrightarrow\Tr_X$ is an embedding of lattices of the same rank~$r$.
It follows that
$$
[\Tr_X:\Tr_Y(n)]^2=\frac{\disc \Tr_Y(n)}{\disc\Tr_X}=n^r \frac{\disc \Tr_Y}{\disc \Tr_X}.
$$
Assuming that $\TransQY$ is isometric to $\TransQX$, equation~\eqref{trivial} gives:
$$
\disc \Tr_Y \equiv \disc \Tr_X \mod (\IQ^*)^2.
$$
Combining these two formulas we deduce that $n^r$ is a square in $\IQ^*$. Since by hypothesis $n$ is not a square in $\IQ^*$, this can happen only if $r$ is even.
\end{proof}

\begin{proposition}\label{subsp}
Let $\gamma\colon X\dashrightarrow Y$ be an isogeny of degree $n$ between two complex projective K3 surfaces. Then there exists an isometry between $\TransQY$ and $\TransQX$ if and only if $\TransQY$ is isometric to $\TransQY(n)$.
\end{proposition}

\begin{proof}
As explained in Section \ref{corresp}, the morphism $\gamma^\ast_\Tr\colon\Tr_Y\to\Tr_X$ is a dilation with scale factor $n$, hence
it is an isometry between $\TransQY(n)$ and~$\TransQX$ (that we still denote $\gamma^\ast_\Tr$ by abuse of notation):
$$
\xymatrix{\TransQY\ar[rr]^{\gamma^\ast_\Tr}\ar[dr]_{\id}& & \TransQX\\
& \TransQY(n)\ar[ur]_{\gamma^\ast_\Tr}}
$$
If there exists an isometry between $\TransQY$ and $\TransQY(n)$, then composed with $\gamma^\ast_\Tr$ it gives an isometry between $\TransQY$ and $\TransQX$, and \viceversa.
\end{proof}

\begin{remark}\label{sym}
The statement is symmetric in $X$
and $Y$, it is equivalent to ``$\TransQX(n)$ is isometric to 
$\TransQX$''. Indeed, $\TransQY(n)\cong \TransQX$ if and only if ${\TransQY\cong \TransQX(n)}$ since $\TransQX(n^2)\cong\TransQX$.
\end{remark}

For any regular rational quadratic form $\varphi$ and any prime number $q$, the parity of the $q$-adic valuation of the determinant of~$\varphi$ is an invariant of the Witt-equivalence class of~$\varphi$. We thus define
$$
\delta_q\colon W(\IQ)\to \IZ/2\IZ,\quad\delta_q(\varphi)= \nu_q(\det(\varphi)) \mod 2.
$$
In particular, expressions like $s^{\delta_q(\varphi)}$, with $s \in \IK^*$, are well defined in $\IK^*/(\IK^*)^2$.
Recall that for any rational quadratic form
 $\varphi = \langle a_1, \ldots, a_n \rangle$ and any $m\in\IZ^\ast$, we put $\varphi(m) := \langle m a_1, \ldots, m a_n \rangle$.
The following proposition proves assertion~(\ref{item2}) of Theorem~\ref{witt} by taking $\varphi$ equal to the rational quadratic form of $\TransQY$.

\begin{proposition}\label{due}
Let $\varphi$ be a regular rational quadratic form of dimension~$n$. Then $\varphi$~and~$\varphi(2)$ are equivalent if and only if $n$ is even and
 for every prime number~$q$ congruent to $3$ or $5$ modulo $8$, one has $\delta_q(\varphi)$=0.
\end{proposition}

\begin{proof}
We can assume, up to equivalence, that 
$$
\varphi = \langle a_1,\ldots, a_{m_1}, 2\,b_1, \ldots, 2 \, b_{m_2} \rangle, 
$$
where $m_1 + m_2 = n$ and $a_i, b_j$ are square-free odd integers. Then $\varphi(2)$ is equivalent to
$$
 \psi := \langle 2\, a_1,\ldots, 2\, a_{m_1}, b_1 \ldots, b_{m_2} \rangle.
$$
Clearly, $\varphi$ and $\psi$ have the same signature. Moreover,
$$
\bar\partial_2 ([\varphi]_W) = \langle \bar 1 \rangle_W^{\oplus m_2} \quad \text{and} \quad
\bar\partial_2 ([\psi]_W) = \langle \bar 1 \rangle_W^{\oplus m_1};
$$
hence, by Proposition~\ref{prop:finite-field} we have that $\bar\partial_2 ([\varphi]_W)$ is Witt-equivalent to $\bar \partial_2 ([\psi]_W)$ if and only if $m_1$ and $m_2$ have the same parity, or equivalently if $n$ is even. 

Let $q$ be an odd prime number. Obviously, $\bar \partial_q([\varphi]_W)$ and $\bar \partial_q([\psi]_W)$ have the same dimension since $\partial_q(2\,r)=2\partial_q(r)$ for all $r \in \IQ^*$. Their discriminants are
\begin{align*}
\Delta(\bar\partial_q ([\varphi]_W))&=(-1)^{\frac{n(n-1)}{2}}2^{n_1}\det(\bar\partial_q(\langle a_1,\ldots, a_{m_1},b_1\ldots, b_{m_2}  \rangle))\\
\Delta(\bar\partial_q ([\psi]_W))&=(-1)^{\frac{n(n-1)}{2}}2^{n_2}\det(\bar\partial_q(\langle a_1,\ldots, a_{m_1},b_1\ldots, b_{m_2}  \rangle))
\end{align*}
where $n_1=\nu_q\left(\prod_{j=1}^{m_2}b_j\right)$ and $n_2=\nu_q\left(\prod_{i=1}^{m_1}a_i\right)$.
These two discriminants are equal in $\IF_q^*/(\IF_q^*)^2$ if and only if $2^{n_1-n_2}$ is a square in $\IF_q$. If $q\equiv\pm 1\mod 8$ this is true by
the second supplement to the law of quadratic reciprocity.
Otherwise $q$ is congruent to $2$ or $3$ modulo $8$ and $2^{n_1-n_2}$ is a square
in $\IF_q$ if and only if $n_1-n_2$ is even, or equivalently if $\delta_q(\varphi)=n_1+n_2\mod 2\equiv 0 \mod 2$. 

Under these conditions $\bar\partial_q ([\varphi]_W)$ and $\bar\partial_q ([\psi]_W)$ are Witt-equivalent for any $q$ by Proposition~\ref{prop:finite-field}, hence $\varphi$ and $\varphi(2)$ are
equivalent by Proposition~\ref{prop:QQ-equivalence}. 
\end{proof}

\begin{proposition}\label{primo}
Let $\varphi$ be a regular rational quadratic form of dimension $n$ and let $p >2$ be a prime number. Then $\varphi$ and $\varphi(p)$ are equivalent if and only if $n$ is even and the following conditions are satisfied
\begin{enumerate}
\item for every prime number $q >2$, $q \neq p$, such that $p$ is not a square in $\IF_q$, one has $\delta_q(\varphi)=0$;
\item the following equation holds in $\IF_p^*/(\IF_p^*)^2$:
$$
\res_p (\det(\varphi)) = (-1)^{\frac{n^2-n}{2} + \delta_p(\varphi)}.
$$
\end{enumerate}
\end{proposition}

\begin{proof}
Up to equivalence we can assume that
$$
\varphi = \langle a_1, \ldots, a_{m_1}, p b_1, \ldots, p b_{m_2} \rangle,
$$
where $m_1+m_2=n$ and $a_i$, $b_j$ are square-free integers not divisible by $p$. Then $\varphi(p)$ is equivalent to
$$
\psi = \langle p a_1, \ldots, p a_{m_1}, b_1, \ldots, b_{m_2} \rangle.
$$
Clearly $\varphi$ and $\psi$ have the same signature and $\bar \partial_2([\varphi]_W) = \bar \partial_2([\psi]_W)$ since $p>2$.

Let $q$ be an odd prime number different from $p$. By a similar argument as in the proof of Proposition~\ref{due} we see that $\bar \partial_q([\varphi]_W)$ and $\bar \partial_q([\psi]_W)$ have the same dimension and that their discriminants are automatically equal when $p$ is a square in $\IF_q^*$. Otherwise they are equal if and only if ${\delta_q(\varphi)}=0$. 

Finally for $q=p$ we have
$$
\bar \partial_p ([\varphi]_W) = \langle \bar b_1, \ldots, \bar b_{m_2} \rangle_W \quad \text{and} \quad
\bar \partial_p ([\psi]_W) = \langle \bar a_1, \ldots, \bar a_{m_1} \rangle_W.
$$
The dimensions $m_1$ and $m_2$ of these forms have the same parity if and only if $n$ is even. Their discriminants are equal if and only if the following equation holds in $\IF_p^*/(\IF_p^*)^2$:
$$
(-1)^{\frac {m_1(m_1-1)}2} \prod_{i=1}^{m_1} a_i = (-1)^{\frac {m_2(m_2-1)}2} \prod_{j=1}^{m_2} b_j
$$
or equivalently
\begin{equation} \label{eq:disc}
 \prod_{i=1}^{m_1} a_i\prod_{j=1}^{m_2} b_j = (-1)^{\frac {m_1(m_1-1)+m_2(m_2-1)}2}.
\end{equation}

We have $\det(\varphi) = p^{m_2} \prod_{i=1}^{m_1} a_i \prod_{j=1}^{m_2} b_j$, so
$$
\nu_p(\varphi)=m_2\quad\text{ and }\quad\res_p(\det(\varphi)) = \prod_{i=1}^{m_1} a_i \prod_{j=1}^{m_2} b_j.
$$
Substituting $m_1 = n - m_2$ and using that $(-1)^n =1$ since $n$ is even, we see that equation \eqref{eq:disc} is equivalent to 
$$
\res_p (\det(\varphi)) = (-1)^{\frac{n^2-n}{2} + \delta_p(\varphi)}.
$$
\end{proof}


\section{Isogenies induced by symplectic automorphisms of K3 surfaces}\label{sympl}

Let $X$ be a complex K3 surface with an automorphism $\sigma$ of prime order $p$ leaving invariant the
symplectic two-form of $X$; such an automorphism is called \emph{symplectic}. 
It is well-known (see \cite{Nikulin}) that $p\in\{ 2, 3, 5, 7\}$.
Symplectic automorphisms of order~2, often called \emph{Nikulin involutions}, have been extensively
studied (see \cite[Definition~5.1]{morrison}). Since the local action of $\sigma$ at a fixed point is given by
a matrix
$$
\left( 
\begin{array}{cc}
\omega & 0\\
0& \bar \omega
\end{array}
\right),
$$
where $\omega$ is a primitive $p$-th root of the unity, and $\bar{\omega}$ is its complex conjugate,
the automorphism $\sigma$ has only isolated fixed points; we denote  their number by $\lambda$.
By Nikulin \cite[Theorem 4.5]{Nikulin} the number $\lambda$ depends only on the order $p$ and is given by
$\lambda = \frac {24}{p+1}$. Moreover if $X$ is generic in the moduli space of K3 surfaces with a symplectic automorphism of order $p$ then:
$$
\rk\NS(X)=\lambda(p-1)+1,\quad\rk\Tr_X=21-\lambda(p-1).
$$
Since $\lambda$ is even, generically $\rk\Tr_X$ is odd.

The quotient surface $\bar Y:= X/\langle\sigma\rangle$ contains $\lambda$ singular
points of type $A_{p-1}$, which are the images of the fixed points in $X$. 
Therefore its minimal resolution $Y \rightarrow \bar Y$ contains $\lambda$ configurations 
of $(-2)$-curves of type $A_{p-1}$. Since the automorphism~$\sigma$ acts symplectically 
on $X$, the surface $Y$ is again a K3 surface (see \eg \cite[Proof of Theorem~4.5]{Nikulin})
and we get a degree $p$ isogeny
$$
\gamma\colon X\dashrightarrow Y.
$$

We give a precise description of the resolution of indeterminacies $\beta\colon\widetilde X\to X$ of the map $\gamma$ (see diagram (\ref{resol}) of Section~\ref{corresp})
generalizing \cite[Section 3]{morrison}.
Denote by 
$p_1,\ldots, p_{\lambda}$ the singular points on $\bar{Y}$, by $E_i^1,\ldots, E_i^{p-1}$ the $(-2)$-curves of the
$A_{p-1}$ graph on the singularity $p_i$ with the properties
\begin{eqnarray*}
(E_i^j)^2&=&-2,\\
(E_i^j,E_i^{j+1})&=&1, \qquad j=1,\ldots, p-2,\\
(E_i^j,E_i^{k})&=&0, \qquad \mbox{for all}~k\not = j+1.\\
\end{eqnarray*}
Let $M_p$ be the smallest primitive sublattice of the lattice $\HH^2(Y,\IZ)$ that contains all 
the curves $E_i^j$ for $i=1,\ldots, \lambda$ and $j=1,\ldots, p-1$. 
By Nikulin, \cite[Proposition~7.1, Theorem~7.2]{Nikulin} the lattice $M_p$ contains exactly
one divisible class of the form
$\eta/p$ where $\eta$ is a linear combination of the $E_i^j$. Since $\NS(Y)$ is a primitive 
sublattice of $\HH^2(Y,\IZ)$, the lattice $M_p$ is a sublattice of
$\NS(Y)$. An explicit expression for $\eta$ is given 
in \cite[Section 6, 1a), p. 116 and Definition 6.2]{Nikulin} as follows. Put:
$$
 \delta_i:= \sum_{j = 1}^{p-1} j E^j_i,
$$
the class $\eta$ can be written
$$
\eta=\sum_{i=1}^\lambda a_i \delta_i
$$
where 
\begin{align*}
(a_1,\ldots,a_{8})&=(1,\ldots,1)\text{ if }  p=2\\
(a_1,\ldots,a_{6})&=(1,\ldots,1)\text{ if } p=3\\
(a_1,a_2,a_3,a_4)&=(1,1,2,2)\text{ if } p=5\\
(a_1,a_2,a_3)&=(1,2,3)\text{ if } p=7.
\end{align*}
By \cite[I\S17]{bpv} there exists a cyclic 
covering
$\bar\gamma\colon\bar X\to Y$ of order $p$ ramified 
 on $\eta$.
If $p=2$ then $\bar{X}$ is the blow up of $X$ at the fixed points of $\sigma$.
In particular $\bar{X}=\widetilde X$ is smooth and we have a commutative diagram (see \cite{vanGeemen-Sarti} or \cite[Section 3]{morrison}):
$$
\xymatrix{
\widetilde X  \ar[d]_{\widetilde\gamma} \ar[r]^\beta     & X  \ar[d] \ar@(u,r)|-\sigma \ar@{-->}[ld]^\gamma \\
                Y         \ar[r]                   & \bar Y
}
$$
If $p>2$ the situation is more complicated since
$\bar{X}$ has Hirzebruch--Jung singularities over the singular points of $\eta$ 
(see \cite[I\S17, III\S2]{bpv} for their description).
We denote by $\varepsilon\colon\widetilde  X \rightarrow \bar X$ the minimal resolution of the singularities of $\bar{X}$.
The counterimage on the surface $\widetilde X$ of each $A_{p-1}$-configuration has the following dual graph:

\bigskip

$p = 3$: 
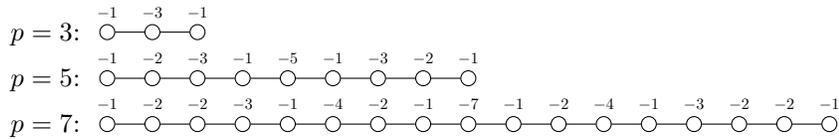
\begin{tikzpicture}[plain node/.style={circle,fill=white!10,draw},every node/.style={scale=0.6}]

\node[plain node] (A) [label=$-1$] {};
\node[plain node] (B) [right of=A,label=$-3$] {};
\node[plain node] (C) [right of=B,label=$-1$] {};

\draw (A) -- (B) -- (C);
\end{tikzpicture}

$p = 5$: 
\begin{tikzpicture}[plain node/.style={circle,fill=white!20,draw},every node/.style={scale=0.6}]

\node[plain node] (A) [label=$-1$] {};
\node[plain node] (B) [right of=A,label=$-2$] {};
\node[plain node] (C) [right of=B,label=$-3$] {};
\node[plain node] (D) [right of=C,label=$-1$] {};
\node[plain node] (E) [right of=D,label=$-5$] {};
\node[plain node] (F) [right of=E,label=$-1$] {};
\node[plain node] (G) [right of=F,label=$-3$] {};
\node[plain node] (H) [right of=G,label=$-2$] {};
\node[plain node] (I) [right of=H,label=$-1$] {};

\draw (A) -- (B) -- (C) -- (D) -- (E) -- (F) -- (G) -- (H) -- (I);
\end{tikzpicture}

$p = 7$: 
\begin{tikzpicture}[plain node/.style={circle,fill=white!10,draw},every node/.style={scale=0.6}]

\node[plain node] (A) [label=$-1$] {};
\node[plain node] (B) [right of=A,label=$-2$] {};
\node[plain node] (C) [right of=B,label=$-2$] {};
\node[plain node] (D) [right of=C,label=$-3$] {};
\node[plain node] (E) [right of=D,label=$-1$] {};
\node[plain node] (F) [right of=E,label=$-4$] {};
\node[plain node] (G) [right of=F,label=$-2$] {};
\node[plain node] (H) [right of=G,label=$-1$] {};
\node[plain node] (I) [right of=H,label=$-7$] {};
\node[plain node] (J) [right of=I,label=$-1$] {};
\node[plain node] (K) [right of=J,label=$-2$] {};
\node[plain node] (L) [right of=K,label=$-4$] {};
\node[plain node] (M) [right of=L,label=$-1$] {};
\node[plain node] (N) [right of=M,label=$-3$] {};
\node[plain node] (O) [right of=N,label=$-2$] {};
\node[plain node] (P) [right of=O,label=$-2$] {};
\node[plain node] (Q) [right of=P,label=$-1$] {};

\draw (A) -- (B) -- (C) -- (D) -- (E) -- (F) -- (G) -- (H) -- (I) -- (J) -- (K) -- (L) -- (M) -- (N) -- (O) -- (P) -- (Q);
\end{tikzpicture}

\bigskip

The self-intersection numbers of the curves can be obtained by a straightforward calculation with continuous fractions following \cite[Theorem III.5.1]{bpv}. For example for $p = 3$ we have a singularity of type $A_{3,1}$ on $\bar{X}$ over the point of intersection of $E_i^1$ and $E_i^2$ for each $i = 1,\ldots,\lambda$. After resolving it one obtains a curve $F_i$ with self-intersection~$-3$. Denote by $\widetilde\gamma\colon\widetilde X\rightarrow Y$ the composition of $\varepsilon$~and~$\bar \gamma$. Since $(\widetilde\gamma^*E_i^j)^2 = 3(E_i^j)^2 = -6$ and $\widetilde\gamma^*E_i^j = 3\tilde E_i^j + F_i$, we get
$$
 -6 = (\widetilde\gamma^*E_i^j)^2 = 9 (\tilde E_i^j)^2 + 6(\tilde E_i^j,F_i) + F_i^2  = 9 (\tilde E_i^j)^2 +6-3
 $$
so the strict transforms $\tilde E_i^j$ have self-intersection~$-1$
In the cases $p = 5$ or $p=7$ the computation is similar.

Contracting these configurations to smooth points we obtain the map $\beta\colon\widetilde X \rightarrow X$, and the following commutative diagram:
\begin{equation} \label{eq:diagram}
\xymatrix{
\bar X \ar[dr]_{\bar \gamma} & \widetilde X  \ar[l]_\varepsilon \ar[d]^{\widetilde\gamma} \ar[r]^\beta     & X         \ar[d] \ar@(u,r)|-\sigma \ar@{-->}[ld] \\
               & Y         \ar[r]                   & \bar Y
}
\end{equation}
Similar diagrams are described by Tan \cite{tan} in a more general setting. 

The map $\beta$ can be easily described by appropriate blow-ups. For instance if $p=3$, to get a smooth quotient $Y$ one has to blow-up three times
each of the fixed points on $X$ (two blow-ups are on the exceptional curve that we obtain after the first blow-up), so that the pullback of the order three automorphism $\sigma$ has no isolated fixed points on the exceptional set and 
 the quotient surface is smooth. The configuration of the exceptional curves after the three blow-ups is the same as in the graph described above
and the surface that we get is clearly $\widetilde{X}$. This works in a similar way for $p=5$ or $p=7$.


\section{A geometric proof of Corollary~\ref{main}}\label{symplnoniso}

Let $\sigma$ be a symplectic automorphism of prime order $p$ on a complex projective surface $X$. Let $Y$ be the minimal resolution
of $X/\sigma$ and let $\gamma\colon X\dasharrow Y$ be the associated isogeny (see Section~\ref{sympl}).
Corollary~\ref{main} is a direct consequence of Theorem~\ref{witt}, or more precisely Proposition~\ref{isog}.
In this section, we give a different proof based upon the geometric properties of the automorphism
generalizing \cite[Proposition 2.5]{vanGeemen-Sarti}.

The invariant sublattice $\HH^2(X,\IZ)^\sigma$ is primitive in $\HH^2(X,\IZ)$, we 
denote 
$$
M:=(\HH^2(X,\IZ)^{\sigma})^\perp
$$
its orthogonal complement in $\HH^2(X,\IZ)$. Since $\sigma$ acts symplectically on~$X$ one has $M \subset \NS(X)$. By \cite[Proposition 10.1]{Nikulin} (see also \cite{vanGeemen-Sarti, Garbagnati-Sarti}) we have 
$$
\rk M = \lambda (p - 1) \quad \text{and} \quad    \disc M = p^\lambda.
$$
In particular, $\rk M$ is always even. Denote
$$
V := \HH^2(X,\IZ)^\sigma\cap\NS(X)=M^\perp\cap\NS(X)
$$ 
the orthogonal complement of $M$ in $\NS(X)$. Using formula \eqref{eq:disc-perp} we obtain
\begin{equation} \label{eq:discV}
\quad p^\lambda \disc V \equiv \disc \NS(X) \mod (\IQ^*)^2.
\end{equation}

Recall (see Section~\ref{sympl}) that $M_p$ denotes the minimal primitive sublattice of $\NS(Y)$ that contains the curves $E^j_i$. Its rank is $\rk M_p=\lambda(p-1)=\rk M$. 
Since $M_p$ is an overlattice of the lattice generated by the curves $E^j_i$, whose discriminant is $p^\lambda$, formula \eqref{formule_disc} gives
$$
    \disc M_p \equiv p^\lambda \mod (\IQ^*)^2.
$$
More precisely by \cite[Proposition 10.1]{Nikulin} we have $\disc M_p=p^{\lambda-2}$.
Put 
$$
W := M_p^\perp\cap\NS(Y) \subset \NS(Y)
$$ 
Note that $\rk W=\rk V$. Using formula \eqref{eq:disc-perp} we get
\begin{equation} \label{eq:discW}
p^\lambda  \disc W \equiv \disc \NS(Y) \mod (\IQ^*)^2.
\end{equation}

Consider the morphism of Hodge structures $\gamma^\ast:=\beta_\ast\widetilde\gamma^\ast\colon\HH^2(Y,\IZ)\to\HH^2(X,\IZ)$ and its restriction to $W$. One has $\HH^2(\widetilde X,\IZ) \cong \HH^2(X,\IZ)\oplus\bigoplus_i \langle C_i \rangle$, where the curves~$C_i$ are
all the curves contracted by $\widetilde\gamma$. The map 
$$
\beta_\ast\colon\HH^2(\widetilde X,\IZ) \cong \HH^2(X,\IZ)\oplus \bigoplus_i\langle C_i \rangle \rightarrow \HH^2(X,\IZ)
$$
is the projection onto the first factor. For any $u\in W$, one has
$$
(\widetilde\gamma^*u, C_i) = (u, \widetilde\gamma_* C_i) = 0
$$
since $\widetilde\gamma_\ast C_i$ is either 0 or one of the curves $E^j_i$. It follows
that $\gamma^\ast$ maps $W$ injectively into $V$. Moreover
$$
\langle\beta_* \widetilde\gamma^* x, \beta_* \widetilde\gamma^* y\rangle_X = \langle\widetilde\gamma^* x, \widetilde\gamma^* y\rangle_{\widetilde X} = p\,\langle x,y\rangle_Y\qquad\forall x,y\in W.
$$
Since $\rk W=\rk V$, it follows that the restriction $\gamma^\ast\colon W_\IQ\to V_\IQ$ is an isomorphism and a dilation with scale factor $p$ of rational quadratic spaces. Thus we have
\begin{equation} \label{eq:discW-discV}
  \disc V \equiv p^{\rk W} \disc W \mod (\IQ^*)^2.
\end{equation}

If $\TransQX$ and $\TransQY$ are isometric, then $\disc \Tr_X \equiv \disc \Tr_Y\mod(\IQ^\ast)^2$. Since the K3 lattice is unimodular, this implies that  $\disc \NS(X) \equiv \disc \NS(Y)\mod (\IQ^\ast)^2$. Putting together equations \eqref{eq:discV}, \eqref{eq:discW} and \eqref{eq:discW-discV} we obtain
$$
p^{\rk W} \equiv 1 \mod (\IQ^*)^2.
$$
Since $\rk W=\rk\NS(X)-\lambda(p-1)$ and $\lambda (p-1)$ is even, we finally get
\[
    p^{\rk\NS(X)} \equiv 1 \mod (\IQ^*)^2.
\]
Since $p$ is prime this is possible if and only if $\rk \NS(X)=\rk \NS(Y)$ is even, or equivalently if $\rk\Tr_X$ is even.
This proves Corollary~\ref{main}.


\section{Examples}\label{exos}

\subsection{Nikulin involutions}

A K3 surface $X$ admits a {\it Shioda--Inose structure} if there is a Nikulin involution $\iota$ on $X$ with rational quotient map 
$\gamma\colon X\dashrightarrow Y$ such that $Y$~is a Kummer surfaces and $\gamma_*$ induces an integral Hodge isometry between $\Tr_X(2)$ and $\Tr_Y$ (see \cite[Definition 6.1]{morrison}). By \cite[Corollary 6.4]{morrison} a K3 surface of Picard number $\rho(X):=\rk\NS(X)$ equal to $19$ or $20$ always admits such a 
structure, so these surfaces come in a natural way with an isogeny of degree $2$ to another K3 surface. If $\rho(X)=19$ then Theorem~\ref{witt} says that there exist no isometry between
$\TransQY$ and $\TransQX$. If $\rho(X)=20$, using Remark~\ref{sym} we see that $\TransQX$ and $\TransQY$ are isometric rational quadratic forms if and only if $\TransQX(2)\cong \TransQX$.
It is not difficult in this case to produce several examples where they are not isometric. Recall that there is a bijection between positive definite even integral matrices of the form 
$$
Q:=\left( 
\begin{matrix}
2a& b\\
b& 2c
\end{matrix}\right), 
$$ 
with $a,b,c,\IZ$, $d:=4ac-b^2>0$, $a,c>0$, $-a\leq b\leq a\leq c$ and transcendental lattices of K3 surfaces of Picard number $20$ (see \cite[Theorem 4]{shiodainose}). Over $\IQ$ the form~$Q$ is equivalent to
the rational quadratic form $\varphi_Q=\langle 2a, \frac{d}{2a}\rangle$. By Proposition~\ref{due}, if the determinant $d=\det(\Tr_X)$ is divisible by a prime number
$q$ congruent to $3$ or $5$ modulo $8$ with an odd exponent then $\TransQX(2)$ is not isometric to $\TransQX$, so by Theorem~\ref{witt} the rational lattice $\TransQX$ is not isometric to $\TransQY$. We give some explicit examples below using well-known K3 surfaces.

\subsection{The Fermat quartic}

Let $F$ be the Fermat quartic $x^4+y^4+z^4+t^4=0$. Its transcendental lattice is given by the quadratic form
$$
\Tr_F = \left(\begin{matrix}
    8 & 0 \\
    0 & 8
  \end{matrix}\right) 
$$
which is equivalent over $\IQ$ to the rational quadratic form $\varphi_F := \langle 2, 2 \rangle$.
We have $\det(\varphi_F)=2^2$ so the $q$-adic valuation of the determinant for $q$ congruent to $3$ or $5$ modulo 8 is always $0$. By Theorem~\ref{witt} we get $\varphi_F \cong \varphi_F(2)$.

The surface $F$ also admits a symplectic automorphism of order three (see \cite{mukaifinite}). We see that $\Tr_{F,\IQ}(3)$ is not isometric to $\Tr_{F,\IQ}$: since $\det(\Tr_F)=2^2\cdot 3^2$, the second condition of Proposition~\ref{primo} can be written:
$$
4=\res_3(\det(\Tr_F))=(-1)^{1+2}=-1,
$$
which is not true in $\IF_3^\ast/(\IF_3^\ast)^2$.

\subsection{The Schur quartic}

Let $S$ be the Schur quartic $x^4-xy^3=z^4-zt^3$. Its transcendental lattice is given by the quadratic form
$$
\Tr_S = \left(\begin{matrix}
    8 & 4 \\
    4 & 8
  \end{matrix}\right)
$$
which is equivalent over $\IQ$ to the rational quadratic form $\varphi_S := \langle 2, 6 \rangle$.
We have $\det(\varphi_S)=3\cdot 2^2$ so the $3$-adic valuation of the determinant is odd. By Theorem~\ref{witt} we get $\varphi_S \not \cong \varphi_S(2)$.

The surface $S$ admits a symplectic automorphism of order three, for instance
$(x:y:z:t)\mapsto (x:\omega y: z:\bar{\omega} t)$ where $\omega$ is a primitive third root of the unity. 
Here $\Tr_{S,\IQ}(3)$ is equivalent to $\langle 6, 18 \rangle\cong \langle 6, 2\rangle$ so it is isometric to $\Tr_{S,\IQ}$. 

\subsection{Surfaces with many nodes}

In the papers \cite{BS} and \cite{iocanad} the authors describe some 
$1$-dimensional families of K3 surfaces obtained as special quotients
of pencils of surfaces in $\IP^3$ that contain surfaces with many nodes.
In particular, looking in \cite[Table 1]{iotransc} we see that there are three 1-dimensional families
such that for all K3 surfaces in the pencil the rational transcendental Hodge structures are no
isometric.

\subsection{Kummer surfaces}

One standard way to produce isogenies that do not come from symplectic automorphisms is to use Kummer surfaces (see \eg \cite{inoseisog}). Let $A$ be a complex abelian surface and $\Gamma$ a subgroup of prime order $p$ of the group $A[p]\cong(\IZ/p\IZ)^{\oplus 4}$ of $p$-torsion points of $A$. Denoting $B:=A/\Gamma$ we have a degree $p$ morphism $g\colon A\lra B$.  
It is easy to see that $g$ induces an order $p$ isogeny between the Kummer surfaces
associated to $A$ and $B$, denoted $\gamma\colon\Km(A)\dashrightarrow\Km(B)$:
$$
\xymatrix{
A  \ar@{.>}[d]\ar[r]^{g}     & B     \ar@{.>}[d]    \\
    \Km(A)       \ar@{.>}[r]^{\gamma}                   & \Km(B)
}
$$
where the vertical arrows are the birational quotients by the involutions $(-\id)$ with 16 fixed points.
 At least for $p>7$ this isogeny can certainly not be induced by a symplectic automorphism of order $p$ on $\Km(A)$ and non-symplectic automorphisms do not produce such maps.

\bibliographystyle{plain}
\bibliography{biblio}

\begin{thebibliography}{10}

\bibitem{bpv}
W.~Barth, K.~Hulek, C.~Peters, and A.~Van~de Ven.
\newblock Compact complex surfaces, 2004.

\bibitem{BS}
W.~Barth and A.~Sarti.
\newblock Polyhedral groups and pencils of {$K3$}-surfaces with maximal
  {P}icard number.
\newblock {\em Asian J. Math.}, 7(4):519--538, 2003.

\bibitem{Beauville}
A.~Beauville.
\newblock The {H}odge conjecture.
\newblock {\em La Matematica}, 2:705--730, 2008.

\bibitem{Buskin}
N.~Buskin.
\newblock Every rational {H}odge isometry between two {$K3$} surfaces is
  algebraic, 2015.
\newblock arXiv:1510.02852.

\bibitem{chen}
X.~Chen.
\newblock Self rational maps of {$K3$} surfaces, 2008.
\newblock arXiv:1008.1619.

\bibitem{galluzzilombardo}
F.~Galluzzi and G.~Lombardo.
\newblock Correspondences between {$K3$} surfaces.
\newblock {\em Michigan Math. J.}, 52(2):267--277, 2004.
\newblock With an appendix by Igor Dolgachev.

\bibitem{Garbagnati-Sarti}
Alice Garbagnati and Alessandra Sarti.
\newblock Symplectic automorphisms of prime order on {$K3$} surfaces.
\newblock {\em J. Algebra}, 318(1):323--350, 2007.

\bibitem{vanGeemenMichigan}
B.~van Geemen.
\newblock Real multiplication on {$K3$} surfaces and {K}uga-{S}atake varieties.
\newblock {\em Michigan Math. J.}, 56(2):375--399, 2008.

\bibitem{vanGeemen-Sarti}
B.~van Geemen and A.~Sarti.
\newblock {Nikulin involutions on {$K3$} surfaces}.
\newblock {\em Math. Z.}, 255(4):731--753, 2007.

\bibitem{inose}
H.~Inose.
\newblock {On certain Kummer surfaces which can be realized as non-singular
  quartic surfaces in $P^3$.}
\newblock {\em {J. Fac. Sci., Univ. Tokyo, Sect. I A}}, 23:545--560, 1976.

\bibitem{inoseisog}
H.~Inose.
\newblock {Defining equations of singular {$K3$} surfaces and a notion of
  isogeny.}
\newblock {Proc. int. Symp. on algebraic geometry, Kyoto, 495--502}, 1977.

\bibitem{ma}
S.~Ma.
\newblock On {$K3$} surfaces which dominate {K}ummer surfaces.
\newblock {\em Proc. Amer. Math. Soc.}, 141(1):131--137, 2013.

\bibitem{morrisonalg}
D.~Morrison.
\newblock Algebraic cycles on products of surfaces, 1984.

\bibitem{morrison}
D.~Morrison.
\newblock On {$K3$} surfaces with large {P}icard number.
\newblock {\em Invent. Math.}, 75(1):105--121, 1984.

\bibitem{morrisonisog}
D.~Morrison.
\newblock Isogenies between algebraic surfaces with geometric genus one.
\newblock {\em Tokyo J. Math.}, 10(1):179--187, 1987.

\bibitem{mukai}
S.~Mukai.
\newblock {On the moduli space of bundles on $K3$ surfaces. I.}
\newblock {Vector bundles on algebraic varieties, Pap. Colloq., Bombay 1984,
  Stud. Math., Tata Inst. Fundam. Res. 11, 341-413}, 1987.

\bibitem{mukaifinite}
S.~Mukai.
\newblock Finite groups of automorphisms of {$K3$} surfaces and the {M}athieu
  group.
\newblock {\em Invent. Math.}, 94(1):183--221, 1988.

\bibitem{Nikulin}
V.~Nikulin.
\newblock Finite groups of automorphisms of {K}\"ahlerian {$K3$} surfaces.
\newblock {\em Trudy Moskov. Mat. Obshch.}, 38:75--137, 1979.

\bibitem{Nikulincorresp}
V.~Nikulin.
\newblock On correspondences between surfaces of {$K3$} type.
\newblock {\em Izv. Akad. Nauk SSSR Ser. Mat.}, 51(2):402--411, 448, 1987.

\bibitem{Nikulinrational}
V.~Nikulin.
\newblock On rational maps between {$K3$} surfaces.
\newblock In {\em Constantin {C}arath\'eodory: an international tribute,
  {V}ol.\ {I}, {II}}, pages 964--995. World Sci. Publ., Teaneck, NJ, 1991.

\bibitem{iocanad}
A.~Sarti.
\newblock Group actions, cyclic coverings and families of {$K3$}-surfaces.
\newblock {\em Canad. Math. Bull.}, 49(4):592--608, 2006.

\bibitem{iotransc}
A.~Sarti.
\newblock Transcendental lattices of some {$K3$}-surfaces.
\newblock {\em Math. Nachr.}, 281(7):1031--1046, 2008.

\bibitem{scharlau}
W.~Scharlau.
\newblock {\em Quadratic and {H}ermitian forms}, volume 270 of {\em Grundlehren
  der Mathematischen Wissenschaften [Fundamental Principles of Mathematical
  Sciences]}.
\newblock Springer-Verlag, Berlin, 1985.

\bibitem{shiodainose}
T.~Shioda and H.~Inose.
\newblock On singular {$K3$} surfaces.
\newblock In {\em Complex analysis and algebraic geometry}, pages 119--136.
  Iwanami Shoten, Tokyo, 1977.

\bibitem{tan}
S.~Tan.
\newblock Surfaces whose canonical maps are of odd degrees.
\newblock {\em Math. Ann.}, 292(1):13--29, 1992.

\end{thebibliography}

\end{document}